\documentclass{article}
\usepackage{amsmath}
\usepackage{amsfonts}
\usepackage{fancyhdr}

\newtheorem{theorem}{Theorem}

\newtheorem{corollary}[theorem]{Corollary}

\newtheorem{definition}[theorem]{Definition}
\newtheorem{example}[theorem]{Example}

\newtheorem{lemma}[theorem]{Lemma}

\newtheorem{proposition}[theorem]{Proposition}
\newtheorem{remark}[theorem]{Remark}

\newenvironment{proof}[1][Proof]{\noindent\textbf{#1.} }{\ \rule{0.5em}{0.5em}}
\numberwithin{equation}{section}

\begin{document}

\title{On space-like generalized constant ratio hypersufaces in Minkowski spaces}

\author{Alev Kelleci\footnote{Adress:F\i rat University, Faculty of Science, Department of Mathematics,23200 Elaz\i\u g/Merkez, Turkey, Phone: (+90)424 237 0000-3708, e-mail:alevkelleci@hotmail.com}, 
Mahmut Erg\" ut\footnote{Adress:Nam\i k Kemal University, Faculty of Science and Letters, Department of Mathematics, 59030 Tekirda\u g/Merkez, Turkey, e-mail:mergut@nku.edu.tr} and Nurettin Cenk Turgay \footnote{Corresponding Author} 
\footnote{Adress:Istanbul Technical University, Faculty of Science and Letters, Department of  Mathematics, 34469 Maslak, Istanbul, Turkey, Phone: (+90)533 227 0041 Fax: (+90)212 285 6386, e-mail:turgayn@itu.edu.tr}}


\maketitle


\begin{abstract}
A hypersurface in a (semi-) Euclidean space $\mathbb{E}^{n+1}_s$ is said to be a
generalized constant ratio (GCR) hypersurface if the tangential part of its
position vector is one of its principle directions. In this work, we move
the study of generalized constant ratio hypersurfaces started in \cite%
{YuFu2014GCRS} into the Minkowski space. First, we get some geometrical
properties of non-degenerated GCR hypersurfaces in an arbitrary dimensional
Minkowski space. In this paper, we study generalized constant ratio (GCR) hypersurfaces in Euclidean
spaces. We mainly focus on the hypersurfaces in $\mathbb{E}^{4}_1$. Then, 
we obtain the complete classification of space-like GCR hypersurfaces with vanishing Gauss-Kronecker
curvature. We also give some explicit examples.


\end{abstract}



\section{Introduction}

The theory of semi-Riemannian submanifolds of semi-Euclidean spaces is
a very active research field. In particular, problems related with position vector of
submanifolds have cough interest of many geometers so far. In this
direction, the notion of constant ratio submanifolds in Euclidean spaces was
firstly introduced by B.-Y. Chen in \cite{ChenCRSurf2001}(see also \cite{Boyadzhiev2007}). By the definition,
a submanifold of Euclidean space is said to be of \textit{constant ratio} if
the ratio of the length of the tangential and normal components of its
position vector is constant. 

Let $M$ be a hypersurface in the Euclidean $\mathbb E^{n+1}$ with the position vector $x$ and $%
\theta$ denote the angle function between $x$ and the unit normal vector
field $N$ of $M$. If the tangential part $x^T$ of $x$ is one of its principal
directions, then $M$ is said to be a generalized constant ratio (GCR)
surface, \cite{YuFu2014GCRS}. One can show that being GCR of $M$ is
equivalent to $Y(\theta)=0$, whenever $Y$ is a tangent vector field
orthogonal to $x$, \cite{YuFu2014GCRS}. Note that $M$ is a CR hypersurface
if and only if it is a GCR surface satisfying $x^T(\theta)=0$.

We also would like to note that  GCR surfaces in Euclidean 3-space $\mathbb{E}^{3}$ are
also related with \textit{constant slope surfaces} introduced by M. I.
Munteanu in \cite{Munteanuconstantslope}, where the author obtain the
complete classification of such surfaces in the $\mathbb{E}^{3}$. Further,
similar techniques are used in \cite{YufuYangspacelikeslope,YufuWangtimelikeslope} in order to obtain the complete classification of
constant slope surfaces in $\mathbb{E}_{1}^{3}$. We would like to note that
an important property of constant slope surfaces in Euclidean 3-space $%
\mathbb{E}^{3}$ and Minkowski 3-spaces is the following: Let $U$ and $x$
denote the projection of position vector on the tangent plane of the surface
and a generic point in ambient space, respectively. If the projection $U$
makes constant angle with the normal vector of the surface at that point,
then $U$ is a canonical principal direction of the surface with the
corresponding principal curvature being different from zero.

Several classifications of GCR hypersurfaces  in semi-Euclidean spaces have been appeared so far. For example, GCR surfaces in the Euclidean space $\mathbb E^3$ and Minkowski space $\mathbb E^3_1$ were classified in \cite{YuFu2014GCRS} and \cite{YuFu2016LGCR}, respectively. Furthermore, in \cite{YuFu2016LGCR} authors also gave the characterizations of the flat and CMC Lorentz GCR surfaces in $\mathbb{E}_{1}^{3}$. 
On the other hand, several characterization results on GCR hypersurfaces of Euclidean spaces were obtained in \cite{TurgayGCRHypEucl}. 

In the present paper, we would like to move the study of GCR hypersurfaces
in Euclidean spaces initiated in \cite{YuFu2014GCRS,TurgayGCRHypEucl} into
semi-Euclidean spaces by obtaining the complete classification of space-like GCR hypersurfaces
in Minkowski 4-space. This paper is organized as follows. In Sect. 2, we
introduce the notation that we will use and give a brief summary of basic
definitions in theory of submanifolds of semi-Euclidean spaces. In Sect. 3,
we obtain some of geometrical properties of space-like GCR hypersurfaces in
a arbitrary dimensional Minkowski space $\mathbb{E}_{1}^{n+1}$. In Sect. 4,
we obtain the complete classification of space-like GCR hypersurfaces in the
Minkowski 4-space.

\section{Preliminaries}

Let $\mathbb{E}^m_s$ denote the pseudo-Euclidean $m$-space with the
canonical pseudo-Euclidean metric tensor $g$ of index $s$ given by 
\begin{equation*}
\widetilde g=\langle\ ,\ \rangle=-\sum\limits_{i=1}^s
dx_i^2+\sum\limits_{j=s+1}^m dx_j^2,
\end{equation*}
where $(x_1, x_2, \hdots, x_m)$ is a rectangular coordinate system in $%
\mathbb{E}^3_1$. We put 
\begin{eqnarray}
\mathbb{S}^{m-1}_{s}(r^{-2})&=&\{x\in\mathbb{E}^m_s: \langle x, x
\rangle=r^{2}\},  \notag \\
\mathbb{H}^{m-1}_{s-1}(r^{-2})&=&\{x\in\mathbb{E}^{n+1}_1: \langle x, x
\rangle=-r^{2}\}.  \notag
\end{eqnarray}
Note that $\mathbb{S}^{m-1}_{s} (r^2)$ and $\mathbb{H}^{m-1}_{s-1}(-r^2)$
are the complete pseudo-Riemannian manifolds of constant curvature $r^2$ and 
$-r^{2}$, respectively. Moreover, we will put $\mathbb{H}^{m-1}_{0}(-r^2)=%
\mathbb{H}^{m-1}(-r^2)$.

A non-zero vector $v$ in $\mathbb{E}^m_s$ is said to be space-like, time-like
and light-like (null) regarding to $\left\langle v,v\right\rangle >0$ , $%
\left\langle v,v\right\rangle <0$ and $\left\langle v,v\right\rangle =0$,
respectively. Note that $v$ is said to be causal if it is not space-like.

\subsection{ Space-like Hypersurfaces in the Minkowski space.}\label{SectHyprPrel}

Let $M$ be an oriented hypersurface in $\mathbb{E}_{1}^{n+1}$ with the position vector $x$ and unit
normal vector $N$ associated with the orientation of $M$. The immersion $x$
(or, equivalently hypersurface $M$) is said to be space-like (resp.
time-like) if the induced metric $g=\left.\widetilde g\right|_{M}$ of $M$ is
Riemannian (resp. Lorentzian). This is equivalent to being time-like (resp.
space-like) of $N$ at each point of $M$.

We denote the Levi-Civita connections of $M$ and $\mathbb{E}_{1}^{n+1}$ by $%
\nabla $ and $\widetilde{\nabla }$, respectively. Then, Gauss and Weingarten
formulas are given, respectively, by 
\begin{eqnarray}
\widetilde{\nabla }_{X}Y &=&\nabla _{X}Y+h\left( X,Y\right) ,
\label{WeingrFormula} \\
\widetilde{\nabla }_{X}N &=&-S(X)
\end{eqnarray}%
for any tangent vector fields $X,Y$, where $h$ and $S$ are the second
fundamental form and the shape operator (or Weingarten map) of $M$,
respectively. The second fundamental form and the shape operator are related
by 
\begin{equation}
\left\langle S(X),Y\right\rangle =\left\langle h\left( X,Y\right)
,N\right\rangle.  \label{hARelatedBy}
\end{equation}%

Now, let $M$ be space-like. Then, its shape operator $S$
is diagonalizable, i.e., there exists a local orthonormal frame field $%
\{e_{1},e_{2},\hdots,e_{n};N\}$ such that $Se_{i}=k_{i}e_{i},\quad i=1,2,%
\hdots,n$. In this case, the vector field $e_{i}$ and smooth function $k_{i}$
are called a principal direction and a principal curvature of $M$.

The Gauss and Codazzi equations are given, respectively, by 
\begin{eqnarray}
\langle R(X,Y)Z,W\rangle  &=&\langle h(Y,Z),h(X,W)\rangle -\langle
h(X,Z),h(Y,W)\rangle ,  \label{MinkGauss} \\
(\widetilde{\nabla }_{X}h)(Y,Z) &=&(\widetilde{\nabla }_{Y}h)(X,Z),
\label{MlinkCodazzi}
\end{eqnarray}%
where $R$ is the curvature tensor associated with the connection $\nabla $
and $\widetilde{\nabla }h$ is defined by 
\begin{equation*}
(\widetilde{\nabla }_{X}h)(Y,Z)=\nabla _{X}^{\perp }h(Y,Z)-h(\nabla
_{X}Y,Z)-h(Y,\nabla _{X}Z).
\end{equation*}


\section{Hypersurfaces in $\mathbb E^{n+1}_1$}

In this section, we consider GCR hypersurfaces in a Minkowski space $\mathbb E^{n+1}_1$.

Let $M$ be a hypersurface in a semi-Euclidean space $\mathbb{E}^{n+1}_s$ and 
$x:M\rightarrow\mathbb{\mathbb{E}}^{n+1}_s$ an isometric immersion. Since $x$
can be considered as a vector field defined on $M$, it can be expressed as 
\begin{equation}  \label{PosVectDecomp1}
x=x^T+x^\perp,
\end{equation}
where $x^T$ and $x^\perp$ denote the tangential and normal parts of $x$.
Before we proceed, we would like to recall the following definition. Note that this definition is given in \cite{YuFu2014GCRS} when the ambient space
is $\mathbb{E}^3$.

\begin{definition}
Let $M$ be a non-degenerated hypersurface in $\mathbb{E}^{n+1}_s$. $M$ is said to be a
generalized constant ratio (GCR) hypersurface if the tangential part of its
position vector is one of its principal directions.
\end{definition}

\begin{remark}
We want to note that if $M$ is a surface in $\mathbb S^2(1)\times\mathbb E$ or $\mathbb H^2(-1)\times\mathbb E$, then $U=x^T$ is called the canonical principal direction of $M$ by some geometers provided $U$ to be  an eigenvector of the shape operator $S$ of $M$, \cite{MunteanuNistor2011,Nistor2013}.
\end{remark}

\begin{remark}\label{RemarkxNorm}
If $x^T=0$ in the decomposition \eqref{PosVectDecomp1}, i.e., $x$ is normal to $M$, then we have $\langle x,x\rangle=\mbox{const}$ which yields that $M$ is an open part of either $\mathbb S^n(r^{-2})$ or $\mathbb H^n(-r^{-2})$ for some $r>0$. 
\end{remark}

Before we proceed, we would like to give the following theorem for the case of being light-like of $x^T$.
\begin{theorem}\label{TheoremxTLightlike}                                         
Let $M$ be a non-degenerated hypersurface in $\mathbb{E}^{n+1}_1$ with position vector $x$. If $M$ is GCR, then the tangential part $x^T$ of $x$ can not be light-like. 
\end{theorem}
\begin{proof}
Consider a non-degenerated hypersurface in $\mathbb{E}^{n+1}_1$ such that $x^T$ is light-like. Then, we have $M$ is time-like and
\begin{equation}  \label{PosVectDecompLightLike}
x=f_1+\langle x,x\rangle N
\end{equation}
for a light-like tangent vector field $f_1$.
                                                                                            
Towards contradiction, assume that $M$ is GCR, i.e., $f_1$ is an eigenvector of $M$. Then, we have $Sf_1=k_1f_1$ for a smooth function $k_1$. Moreover, there exists a light-like tangent vector field $f_2$ such that $\langle f_1,f_2\rangle=-1$ and $\langle Sf_1,f_2\rangle=-k_1$ which implies $h(f_1,f_2)=-k_1N$.
By applying $f_2$ to \eqref{PosVectDecompLightLike} and considering $\widetilde\nabla_{f_2}x=f_2$, we obtain
\begin{equation}  \label{PosVectDecompLightLike2}
f_2=\nabla_{f_2}f_1+h(f_1,f_2)N+f_2\left(\langle x,x\rangle\right) N-\langle x,x\rangle Sf_2.
\end{equation}
The normal part of this equation gives
$$k_1=f_2\left(\langle x,x\rangle\right)=-2.$$
However, by a further computation using \eqref{PosVectDecompLightLike2}, we get
$$ -1=  \langle x,x\rangle \langle Sf_2,f_1\rangle$$
which implies being constant of $\langle x,x\rangle$. Hence, $M$ is an open part of either $\mathbb S^n_1(r^{-2})$ or $\mathbb H^n(-r^{-2})$ for some $r>0$ which yields a contradiction.
\end{proof}

\begin{remark}
Because of Remark \ref{RemarkxNorm} and Theorem \ref{TheoremxTLightlike}, after this point, we, locally, assume that $x^T \neq0$.
\end{remark}

We also need the following lemma given in \cite{ChenTNsub}.

\begin{lemma} \label{intd}
Let $x: M\longrightarrow \mathbb E^m_{\upsilon}$ be an isometric immersion of a Riemannian n-manifold
into the pseudo-Euclidean space $\mathbb E^m_{\upsilon}$. Then, on the open subset
$U=\left\{p \in M: x^T\neq 0\right\}$ the distribution $\mathbf D=\left\{X \in T_pU: \left\langle X,x^T\right\rangle=0\right\}$
is an integrable distribution, \cite{ChenTNsub}.
\end{lemma}

Now, we would like to give the following result obtained from the above Lemma \ref{intd}, directly.

\begin{remark} \label{DDbotint}
Let $M$ be a space-like GCR hypersurface in $\mathbb{E}_{1}^{n+1}$ Minkowski spaces. 
Then, $\mathbf D=Span\left\{{e_2,\hdots,e_n}\right\}$ and $\mathbf D^{\bot}=Span\left\{{e_1}\right\}$ 
are integrable distributions on $M$.
\end{remark} 

Now, we will obtain with the geometrical properties of GCR hypersurfaces given in the following proposition.

\begin{proposition}
\label{PROPOPP2Ext} Let $M$ be an oriented hypersurface in the Minkowski
space $\mathbb{E}^{n+1}_1$ and $x$ its position vector. Consider a unit tangent vector field 
$\displaystyle e_1=\frac{x^T}{\varepsilon(\langle x^T,x^T\rangle)^{1/2}}$ along $x^T$. Then, $M$ is
a GCR hypersurface if and only if a curve $\alpha$ is a geodesic of $M$
whenever it is an integral curve of $e_1$.
\end{proposition}

\begin{proof}
We will consider being space-like or time-like of $x^T$, separately.

\textbf{Case I.} Let $x^{T}$ is time-like. In this case, $%
e_{1}=x^{T}/(-\langle x^{T},x^{T}\rangle )^{1/2}$ is time-like and $M$ is
Lorentzian. Thus, we have 
\begin{equation*}
x=-\langle x,e_{1}\rangle e_{1}+\langle x,N\rangle N.
\end{equation*}%
Since $\widetilde{\nabla }_{e_{1}}x=e_{1}$, this equation yields 
\begin{equation*}
e_{1}=(1-\langle x,N\rangle \langle Se_{1},e_{1}\rangle )e_{1}-\langle
x,e_{1}\rangle \widetilde{\nabla }_{e_{1}}e_{1}+\langle x,Se_{1}\rangle
N-\langle x,N\rangle Se_{1}.
\end{equation*}%
The tangential part of this equation yields $Se_{1}=k_{1}e_{1}$ if and only
if $\nabla _{e_{1}}e_{1}=0$ which is equivalent to being geodesic of all
integral curves of $e_{1}$.

\textbf{Case II.} Let $x^T$ is space-like. In this case, $e_1=x^T/(\langle
x^T,x^T\rangle)^{1/2}$ is space-like. Thus, we have 
\begin{equation}  \label{PROPCASEIIDECOMOP}
x=\langle x,e_1\rangle e_1+\varepsilon\langle x,N\rangle N,
\end{equation}
where $\varepsilon$ is either 1 or -1 regarding to being time-like or
space-like of $M$, respectively.

Similar to Case I, we obtain $Se_{1}=k_{1}e_{1}$ if and only if $\nabla
_{e_{1}}e_{1}=0$. Consequently, the proof is completed.
\end{proof}

Now, assume that $M$ is an oriented space-like GCR hypersurface in $\mathbb{E}%
_{1}^{n+1}$, $x$ its position vector satisfies the condition $\mu=\sqrt{\left|\left\langle x,x\right\rangle\right|}$
 and $\{e_1,e_2,$ $\hdots,e_n;N\}$ is a local orthonormal frame
field consisting of principal directions of $M$, $k_1,k_2,\hdots,k_n$ are
corresponding principal curvatures and $e_1$ is proportional to $x^T$ and
 $N$ is the unit normal vector field associated with the orientation of $M$. 
Note that $N$ is time-like because of the hypersurface $M$ being a space-like one. Further, let
 $\left\langle x^T,x^T\right\rangle \neq0$ and $e_1=x^T/(\langle
x^T,x^T\rangle)^{1/2}$. Also we can locally assume 
$\mu \neq0$, since $M$ is non-degenerated. Thus, the position vector $x$ of $M$ satisfies either $\langle
x,x\rangle<0$ or $\langle x,x\rangle>0$.

\textbf{Case I}. Let $\langle x,x\rangle=\mu^2$. In this case, %
\eqref{PosVectDecomp1} implies 
\begin{equation}  \label{GCRE14DexompofxCaseI}
x=\mu \mathrm{cosh}\theta e_{1}-\mu \mathrm{sinh}\theta N
\end{equation}%
and from the assumption we have
\begin{subequations}
\label{GCRE14DexompofxCasenueqsALL}
\begin{eqnarray}
e_{1}(\mu ) &=&\mathrm{cosh}\theta , \\
e_{j}(\mu ) &=&0,\quad j=2,3,\hdots,n.
\end{eqnarray}%
One consider the equalities \eqref{GCRE14DexompofxCasenueqsALL} and 
$\langle e_1,e_1\rangle =1$ in the decomposition \eqref{GCRE14DexompofxCaseI}, then  
\end{subequations}
\begin{subequations}
\label{GCRE14DexompofxCase1Eq1ALL}
\begin{eqnarray}
\nabla _{e_{1}}e_{1} &=&0,  \label{GCRE14DexompofxCase1Eq1ALLEQ1} \\
e_{1} &=&(-\mathrm{cosh}^{2}\theta +\mu \mathrm{sinh}\theta e_{1}(\theta
)+k_{1}\mu \mathrm{sinh}\theta )e_{1}  \notag \\
&&(-k_{1}\mu \mathrm{cosh}\theta -\mathrm{sinh}\theta \mathrm{cosh}\theta
-\mu \mathrm{cosh}\theta e_{1}(\theta ))N, \\
e_{j} &=&\mu \mathrm{sinh}\theta e_{j}(\theta )e_{1}+\mu \mathrm{cosh}\theta
\nabla _{e_{j}}e_{1}  \notag \\
&&-\mu \mathrm{cosh}\theta e_{j}(\theta )N+k_{j}\mu \mathrm{sinh}\theta
e_{j},\quad j=2,3,\hdots,n
\end{eqnarray}%
are obtained. From there, we obtain
\end{subequations}
\begin{subequations}
\label{GCRE14DexompofxCase1Eq2ALL}
\begin{eqnarray}
k_{1} &=&-e_{1}(\theta )-\frac{\mathrm{sinh}\theta }{\mu },
\label{GCRE14DexompofxCase1Eq2a} \\
e_{j}(\theta ) &=&0,  \label{GCRE14DexompofxCase1Eq2b} \\
\nabla _{e_{j}}e_{1} &=&\frac{1-k_{j}\mu \mathrm{sinh}\theta }{\mu \mathrm{%
cosh}\theta }e_{j}, \quad j=2,3,\hdots,n.
\end{eqnarray}%
\end{subequations}

\textbf{Case II}. Let $\langle x,x\rangle=-\mu^2$. In this case, \eqref{PosVectDecomp1}%
\label{PageCaseI} implies 
\begin{equation}  \label{GCRE14DexompofxCaseII}
x=\mu \mathrm{sinh}\theta e_{1}-\mu \mathrm{cosh}\theta N.
\end{equation}%
Since $\langle x,x\rangle =-\mu ^{2}$, we get
\begin{subequations}
\label{GCRE14DexompofxCaseIInueqsALL}
\begin{eqnarray}
e_{1}(\mu ) &=&-\mathrm{sinh}\theta , \\
e_{l}(\mu ) &=&0,\quad l=2,3,\hdots,n.
\end{eqnarray}%
By considering the equalities \eqref{GCRE14DexompofxCaseIInueqsALL} and $\langle e_1,e_1\rangle =1$
 in \eqref{GCRE14DexompofxCaseII}, we obtain 
\end{subequations}
\begin{subequations}
\label{GCRE14DexompofxCase2Eq1ALL}
\begin{eqnarray}
\nabla _{e_{1}}e_{1} &=&0,  \label{GCRE14DexompofxCase2Eq1ALLEQ1} \\
e_{1} &=&(-\mathrm{sinh}^{2}\theta +\mu \mathrm{cosh}\theta e_{1}(\theta
)+k_{1}\mu \mathrm{cosh}\theta )e_{1}  \notag \\
&&(-k_{1}\mu \mathrm{sinh}\theta -\mathrm{sinh}\theta \mathrm{cosh}\theta
-\mu \mathrm{sinh}\theta e_{1}(\theta ))N, \\
e_{l} &=&\mu \mathrm{cosh}\theta e_{l}(\theta )e_{1}+\mu \mathrm{sinh}\theta
\nabla _{e_{l}}e_{1}  \notag \\
&&-\mu \mathrm{sinh}\theta e_{l}(\theta )N+k_{l}\mu \mathrm{cosh}\theta
e_{l},\quad l=2,3,\hdots,n.
\end{eqnarray}%
So, we have
\end{subequations}
\begin{subequations}
\label{GCRE14DexompofxCase2Eq2ALL}
\begin{eqnarray}  \label{GCRE14DexompofxCase2Eq2c}
k_{1} &=&-e_{1}(\theta )-\frac{\mathrm{cosh}\theta }{\mu },
\label{GCRE14DexompofxCase2Eq2a} \\
e_{l}(\theta ) &=&0,  \label{GCRE14DexompofxCase2Eq2b} \\
\nabla _{e_{l}}e_{1} &=&\frac{1-k_{l}\mu \mathrm{cosh}\theta }{\mu \mathrm{%
sinh}\theta }e_{l},\quad l=2,3,\hdots,n.
\end{eqnarray}%

Note that since $M$ is a space-like hypersurface, so $S$ its shape operator can be diagonalized, 
i.e., $S(e_i)=k_i e_i$,  $i=1,\hdots,n$. Therefore, the Codazzi equations
\eqref{MlinkCodazzi} can be written as
\end{subequations}
\begin{equation}  \label{jij}
e_i(k_j) = \omega_{ij}(e_j)(k_i-k_j), \quad i,j=1,\hdots,n
\end{equation}
and 
\begin{equation}  \label{ijk}
\omega_{ij}(e_k)(k_i-k_j) = \omega_{ik}(e_j)(k_i-k_k), \quad i,j,k=1,\hdots,n,
\end{equation}
for each triplet (j,i,j) and (i,j,k), respectively.

By summing up the results obtained so far, we would like to state following
proposition.

\begin{proposition}\label{PROPOPP2ExtVersion2}
Let $M$ be a space-like hypersurface in the Minkowski space $\mathbb{E}%
_{1}^{n+1}$ and $\langle x,x\rangle <0$ (resp. $\langle x,x\rangle >0$).
Then $M$ is GCR hypersurface if and only if $Y(\theta )=0$, whenever $%
\langle Y,x^{T}\rangle =0$, where $\theta $ is the angle function define in %
\eqref{GCRE14DexompofxCaseI} (resp. \eqref{GCRE14DexompofxCaseII}).
\end{proposition}

\begin{proof}
Let $M$ be a space-like GCR hypersurface and $x$ its position vector. In this case,  
the position vector $x$ is the one of the decomposition given in \eqref{GCRE14DexompofxCaseI} or 
\eqref{GCRE14DexompofxCaseII}. The necessary condition follows from \eqref{GCRE14DexompofxCase1Eq2b} (resp. %
\eqref{GCRE14DexompofxCase2Eq2b}), directly. The converse follows from a direct
computation.
\end{proof}

Although it is out of scope of this paper, we would like to state the
following result which is a direct result of \eqref{GCRE14DexompofxCase1Eq2ALL} and %
\eqref{GCRE14DexompofxCase2Eq2ALL}.

\begin{proposition}
Let $M$ be a space-like GCR hypersurface in the Minkowski space $\mathbb{E}%
^{n+1}_1$ and $e_1$ is a unit normal vector field along $x^T$. Then there
exists a local coordinate function $\hat s$ such that $e_1=\partial_{\hat s}$.
\end{proposition}

\begin{proof}
We consider the case $\langle x,x\rangle<0$. The other case follows from an
analogous computation.

Let $\zeta _{1},\zeta _{2},\hdots,\zeta _{n}$ be the dual base of $%
e_{1},e_{2},\hdots,e_{n}$. By a direct computation using %
\eqref{GCRE14DexompofxCase2Eq1ALLEQ1} and \eqref{GCRE14DexompofxCase2Eq2c} , we obtain $%
d\zeta _{1}=0$ , i.e., $\zeta _{1}$ is closed. Poincare Lemma (see in \cite%
{ChenPRGeom2001}) yields that it is exact, i.e., there exists a local
coordinate function $s$ such that $\zeta _{1}=d\hat{s}$.
\end{proof}

\section{Space-like hypersurfaces in $\mathbb E^4_1$}

In this section, we consider space-like GCR hypersurfaces in the Minkowski 4-space with vanishing Gauss-Kronicker curvature. 

Let $M$ be a space-like hypersurface in $\mathbb{E}^4_1$ with the position vector $x$. We will consider the following cases independently: If $\langle x,x\rangle>0$ and $\langle x,x\rangle<0$, then $M$ is said to lay in the space-like cone and light-like cone of $\mathbb E^4_1$, respectively. 

\subsection{Hypersurfaces Lying in the Space-Like Cone}

In this subsection, we deal with the following case: the position vector $x$ of the
 hypersurface $M$ lies in the space-like cone, i.e., $\left\langle x,x\right\rangle =\mu ^{2}$.
 In this case, the position vector $x$ is decomposed as in \eqref{GCRE14DexompofxCaseI}
 for smooth functions $\theta$ and $\mu$ as before defined. By considering the equations 
\eqref{GCRE14DexompofxCasenueqsALL}-\eqref{GCRE14DexompofxCase1Eq2ALL}, %
\eqref{jij} and \eqref{ijk} for $n=3$, we will give the following, directly:

\begin{lemma} \label{Case1ClassThmDiagonSpacelikekClm11}
The Levi-Civita connection, $\nabla$ of the GCR hypersurface $M$ satisfies  
\begin{subequations} \label{CASEISpacelikeLeviCivitaEq11ALL}
\begin{eqnarray}
\label{CASEISpacelikeLeviCivitaEq11a}\nabla _{e_{1}}e_{1}=0, \quad \nabla _{e_{1}}e_{2}=\omega_{23}(e_1)e_3, \quad \nabla_{e_1}e_3=-\omega_{23}(e_1)e_2,
\\\label{CASEILeviCivitaEq11b}
\nabla _{e_{2}}e_{1}= \omega_{12}(e_2)e_2, \quad \nabla _{e_{2}}e_{2}=-\omega_{12}(e_2)e_1+\omega_{23}(e_2)e_3, \quad \nabla_{e_2}{e_3}=-\omega_{23}(e_2)e_2,
\\\label{CASEILeviCivitaEq11c}
\nabla _{e_{3}}e_{1}= \omega_{13}(e_3)e_3, \quad \nabla _{e_{3}}e_{2}=-\omega_{23}(e_3)e_3, \quad \nabla_{e_3}{e_3}=-\omega_{13}(e_3)e_1-\omega_{23}(e_3)e_2.
\end{eqnarray}
\end{subequations}
Here, $\displaystyle \omega_{1l}(e_l)=\frac{1-\mu \sinh \theta k_l}{\mu \cosh \theta}, l=2,3$.
 Also, for $i,j,k=1,2,3$, the principal curvature functions $k_1,k_2,k_3$ satisfy
\begin{subequations} \label{ClassThmDiagonSpacelikekCod1Case11ALL}
\begin{eqnarray} 
\label{ClassThmDiagonSpacelikekCod1Case111a}
\omega_{23} (e_1) (k_2-k_3) =0, &&
\\\label{ClassThmDiagonSpacelikekCod1Case111b}
e_2(k_1) = e_3(k_1) = 0, &&
\\\label{ClassThmDiagonSpacelikekCod1Case111c}
e_1(k_2) = \omega_{12}(e_2)(k_1-k_2), &\quad e_3(k_{2}) = \omega_{23}(e_2)(k_2-k_3),
\\\label{ClassThmDiagonSpacelikekCod1Case111d}
e_1(k_{3}) = \omega_{13}(e_3)(k_1-k_3), &\quad e_2(k_3) = \omega_{23}(e_3)(k_2-k_3). 
\end{eqnarray}
\end{subequations}
\end{lemma}

Before we proceed to our main result, we would like to give the following examples of GCR hypersurfaces
with vanishing Gauss-Kronecker curvature:

\begin{example} \label{GCRE14orn1}
Let $M = \mathbb H^2 \times \mathbb E^1$ in Minkowski spaces $\mathbb E^4_1$ be a hypercylinder parametrized as 
$$x(s,t,u)=\Big(x_1(s,t),x_2(s,t),x_3(s,t),u\Big).$$
Further, its normal is $N(s,t)=\Big(x_1(s,t),x_2(s,t),$ $x_3(s,t),0\Big)$. Therefore, $x$ can be written as
 $\displaystyle x = u \frac{\partial}{\partial u}+ N$. Note that, here the tangent vector $\displaystyle%
\frac{\partial}{\partial u}$ is the principal direction corresponding with the principal curvature $k_1=0$.
Consequently, the hypercylinder $M$ with vanishing Gauss-Kronecker curvature is a GCR hypersurface.
\end{example}

\begin{example} \label{GCRE14orn2}
Let $\alpha$ be a unit speed curve lying on $\mathbb S^3_1(1) \subset \mathbb E^4_1$ 
and $F_1(u),F_2(u)$ two orthonormal vector fields spanning the normal space of $\alpha$ in $\mathbb S^3_1(1)$,
 i.e., 
\begin{eqnarray}
\left\langle F_1,F_2\right\rangle=\left\langle F_1,\alpha \right\rangle=\left\langle F_1,\alpha^{\prime}\right\rangle=\left\langle F_2,\alpha \right\rangle=\left\langle F_2, \alpha^{\prime}\right\rangle=0, \\
\left\langle F_1,F_1\right\rangle=\left\langle F_2,F_2\right\rangle=1.
\end{eqnarray}
Consider the hypersurface in $\mathbb E^4_1$ given by 
\begin{equation} \label{GCRE14x1}
x(s,t,u)=s \alpha(u)-c \Big(F_1(u) \cosh \left(\frac{t}{c}\right)+F_2(u) \sinh \left(\frac{t}{c}\right)\Big)
\end{equation}
for a non-constant $c$. One can check that the unit normal vector
field of $M$ is $N= F_1(u) \cosh \left(\frac{t}{c}\right)+F_2(u) \sinh \left(\frac{t}{c}\right)$
 by a direct computation and principal directions of $M$ are obtained as $e_1= \partial_s =\alpha(u), \quad e_2= \partial_t,%
 \quad e_3= \frac{1}{\left\|x_u\right\|} \partial_u $ corresponding to principal curvatures $0,\frac{1}{c}, k_3$, respectively.
 Therefore, the hypersurface $M$ is a GCR and its Gauss-Kronecker curvature vanishes because of the first principal curvature $k_1=0$. 
\end{example}

\begin{example} \label{GCRE14orn3}
Let $y : \Lambda \longrightarrow \mathbb S^3_1(1) \subset \mathbb E^4_1$
 be an oriented regular surface with the spherical normal $N$, where 
$\Lambda$ is an open subset in $\mathbb R^2$. Consider the hypersurface $M$ in $\mathbb E^4_1$ given by
\begin{eqnarray} 
\nonumber x: I \times \Lambda \longrightarrow \mathbb E^4_1, \\
\label{GCRE14x2} x(s,t,u)=s y(t,u)-c N(t,u)
\end{eqnarray}
where $c$ is a constant. Here, since the vector field $N$ is the spherical normal of the surface $y$, we have 
$$\left\langle y_t,N\right\rangle = \left\langle y_s,N\right\rangle = \left\langle y,N\right\rangle = 0.$$
By considering these in \eqref{GCRE14x2}, we get directly 
$$\left\langle x_t,N\right\rangle = \left\langle x_s,N\right\rangle = \left\langle x_u,N\right\rangle = 0.$$
 So, one can conclude that the vector field $N$ is also the unit normal of the hypersurface of $M$.
 Furthermore, we have $x_{ss}=0$ from \eqref{GCRE14x2} and also
$\left\langle x_{st},N\right\rangle =0$ and $\left\langle x_{tt},N\right\rangle =0$. 
 So, $h(\partial_s,X)=0$ is satisfied for all tangent vector $X$ on $M$ which say that $S(\partial_s)=0$.
 Therefore, the Gauss-Kronecker curvature of the hypersurfece $M$ vanishes and 
$\partial_s=y(t,u)$ is a principal direction of the hypersurface. Consequently, $M$ is a GCR hypersurface.
\end{example}

In the rest of this part, we will prove the following theorem:

\begin{theorem} \label{GCRE14Spacelikekoni}
Let $M$ be a space-like hypersurface with vanishing Gauss-Kronecker curvature
 in the Minkowski space $\mathbb E^4_1$. Then, $M$ is a GCR hypersurface if and only if it is congruent to one of the following 3 types of hypersurfaces.
\begin{enumerate}
\item [(i)] A part of the hypercylinder given in Example \ref{GCRE14orn1},

\item [(ii)] A hypersurface parametrized with \eqref{GCRE14x1} in Example \ref{GCRE14orn2},

\item [(iii)] A hypersurface parametrized with \eqref{GCRE14x2} in Example \ref{GCRE14orn3}.
\end{enumerate}
\end{theorem}

In order to do the proof of Theorem \ref{GCRE14Spacelikekoni}, we will firstly prove the followings:

\begin{proposition} \label{k10}
Let $M$ be a space-like hypersurface in the Minkowski space $\mathbb E^4_1$. If
 the Gauss-Kronecker curvature of $M$ vanishes, then the principal curvature $k_1$ 
of $M$ identically vanishes, i.e., the following is satisfied
\begin{equation} \label{spacelikee1theta}
e_1(\theta)=-\frac{\sinh \theta}{\mu}.
\end{equation}
\end{proposition}

\begin{proof}
Let $M$ be a space-like hypersurface with vanishing Gauss-Kronecker curvature 
in the Minkowski space $\mathbb E^4_1$ and $e_1=x^T/\|x^T\|,e_2,e_3$ its principal 
directions with corresponding principal curvatures $k_1,k_2,k_3$, respectively, at a point of 
$p \in M$. Assume towards contradiction, that $k_1(p) \neq 0$ at the point $p$. In this case
there exist a neighborhood $\mathcal N_p$ of $p$ on which $k_1$ does not vanish.
 Since $M$ is a space-like GCR hypersurface, its shape operator $S$ can be diagonalized. 
 Moreover, since the hypersurface $M$ is a flat, $det S = k_1 k_2 k_3 =0$. 
Note that as $k_1 \neq 0$, one conclude $k_2 = 0$ or $k_3 = 0$ for all point on $\mathcal N_p$. 

Without loss of generality, assume $k_2 = 0$. In this case, we get directly
$\displaystyle \omega_{12}(e_2)=\frac{1}{\mu \cosh \theta}$ for $l=2$, from 
Lemma \ref{Case1ClassThmDiagonSpacelikekClm11}. On the other hand, if we consider the first equation given in
 \eqref{ClassThmDiagonSpacelikekCod1Case111c} with the last result, then we get $\displaystyle 0=\frac{1}{\mu \cosh \theta}k_1$ on
$\mathcal N_p$ which is a contradiction. Thus, we have $k_1=0$ on $M$. Consequently, if this result substitutes in 
\eqref{GCRE14DexompofxCase2Eq2a}, then we obtain \eqref{spacelikee1theta}.
\end{proof}

We also need the following being the result of the above proposition.

\begin{corollary} 
Let $M$ be a space-like hypersurface in the Minkowski space $\mathbb E^4_1$. If
 the Gauss-Kronecker curvature of $M$ vanishes, then the position vector $x$ given in \eqref{GCRE14DexompofxCaseI} 
is decomposed as  
\begin{equation} \label{GCRE14newx}
x = f e_1-c N
\end{equation}
where the smooth function $f$ satisfies
\begin{equation} \label{GCRE14f1}
e_1(f)=1, \quad e_2(f)=e_3(f)=0
\end{equation}
and $c$ is a non-zero constant.  
\end{corollary}

\begin{proof}
Let $M$ be a space-like hypersurface with vanishing Gauss-Kronecker curvature 
in the Minkowski space $\mathbb E^4_1$, i.e., the equation \eqref{spacelikee1theta} 
is satisfied. Considering with together \eqref{spacelikee1theta} and equations given in \eqref{GCRE14DexompofxCasenueqsALL} for $n=3$, 
 then we have
\begin{subequations} \label{e1muEq1ALL}
\begin{eqnarray} 
\label{e1muEq1a} e_1(\mu \cosh \theta)= \cosh \theta \cosh \theta + \mu \sinh \theta \Big(-\frac{\sinh \theta}{\mu}\Big) = 1, &&\\
\label{e1muEq1b} e_1(\mu \sinh \theta)= \cosh \theta \sinh \theta + \mu \cosh \theta \Big(-\frac{\sinh \theta}{\mu}\Big) = 0.
\end{eqnarray} 
\end{subequations}
 Morever, if the last equalities are considered in \eqref{GCRE14DexompofxCaseI} 
with a smooth function $f$ defined by \eqref{GCRE14f1}, then we obtain directly \eqref{GCRE14newx}. 
Therefore, the proof of Corollary is completed. 
\end{proof}

Now, we are ready to start the proof of the first main theorem:

\noindent\textit{The proof of Theorem 2.} The position vector $x$ of a space-like GCR hypersurface 
$M$ with vanishing Gauss-Kronecker curvature is given in \eqref{GCRE14newx}. 
Note that if the equalities \eqref{e1muEq1ALL} consider in 
$\displaystyle \omega_{1j}(e_j)=\frac{1-\mu \sinh \theta k_j}{\mu \cosh \theta}$ 
such that $j=2,3$, then we get 
\begin{equation} \label{GCRE14omeganew}
\omega_{1j}(e_j)=\frac{1-c k_j}{f}, j=2,3
\end{equation}
where $c$ is a non-zero constant. On the other hand, we get $\lbrack e_2, e_3] = -\omega_{23}(e_2)e_2-\omega_{23}(e_3)e_3$
 by considering \eqref{CASEISpacelikeLeviCivitaEq11ALL}.
So, we conclude the distribution $\mathbf{D}^{\bot} = Span\left\{e_2,e_3\right\}$ 
is involutive and the distribution $\mathbf{D }= Span\left\{e_1\right\}$ is trivially involutive. 
Therefore, there exist $(s,t,u)$ local coordinate system 
such that $\mathbf{D }= Span\left\{\partial_s\right\}$ and $\mathbf{D}^{\bot} =
Span\left\{\partial_t,\partial_u\right\}$ by applying local Frobenius' Theorem. 
On the other hand, as before mentioned in Proposition \ref{k10}, we know the
 first principal curvature of the flat hypersurface $M$ vanishes. 
Thus we get $\widetilde{\nabla}_{e_1}{N}=0$ and $\widetilde{\nabla}_{e_1}{e_1}=0$ with direct calculation.
 From there, one conclude $e_1=e_1(t,u)$ and $N=N(t,u)$. Furthermore, we get 
$f=f(s)$ from \eqref{GCRE14f1}. If the obtained expressions are substituted in the decomposition 
\eqref{GCRE14newx}, then the position vector $x$ is given by
\begin{equation}  \label{GCRE14genelnewx}
x(s,t,u) = f(s) e_1(t,u) - c N(t,u).
\end{equation}

Now, we want to consider three cases seperately:

\textbf{Case 1.} Let $\nabla_{e_2}{e_1}=0$ and $\nabla_{e_3}{e_1}=0$. In this case,
we get directly $\omega_{1l}(e_l)=0$ for $l=2,3$, from the first equalities in 
\eqref{CASEILeviCivitaEq11b} and \eqref{CASEILeviCivitaEq11c}, 
. Thus, the equations \eqref{GCRE14omeganew} imply $\displaystyle k_2=k_3=\frac{1}{c}$, i.e., 
$M$ is a space-like isoparametric hypersurface with principal curvatures of 
GCR hypersurface $M$ obtained as $0, 1/c, 1/c$. Thus, $M$ is a part of the hypercylinder
 $\mathbb{H}^2 \times \mathbb{E}^1$ given in Example \ref{GCRE14orn1},
 (see \cite{Nomizuiso}). Consequently, we have the case (i) of the Theorem.

\textbf{Case 2.} Let $\nabla_{e_2}{e_1}=0$ and $\nabla_{e_3}{e_1}\neq 0$.
In this case, we get directly $\omega_{12}(e_2)=0$ and $\omega_{13}(e_3) \neq 0$ from the first
 equalities in \eqref{CASEILeviCivitaEq11b} and \eqref{CASEILeviCivitaEq11c}, respectively.
 If the first of obtained equalities substitutes in \eqref{GCRE14omeganew} for $j=2$,
 then we get $k_2 = \frac{1}{c}$ where $k_2 \neq k_3$. If this expression is considered in 
 \eqref{ClassThmDiagonSpacelikekCod1Case111a} and %
\eqref{ClassThmDiagonSpacelikekCod1Case111c}, then we get $\omega_{23}(e_1)=0$ and $\omega_{23}(e_2)=0$,
 respectively. So, from \eqref{CASEISpacelikeLeviCivitaEq11ALL} 
we get $\lbrack e_2, e_3] = \omega_{23}(e_3)e_3$,
 $\lbrack e_1, e_2] =0$ and $\lbrack e_1, e_3] = -\omega_{13}(e_3)e_3$ with direct calculation. 
On the other hand since $\mathbf{D}= Span\left\{\partial s\right\}$ and $\mathbf{D}^{\bot}=
Span\left\{\partial_t,\partial_u\right\}$, there exist some smooth functions $a_{11}, a_{22}, a_{23}, a_{32}, a_{33}$
 such that
\begin{subequations} \label{eqALL}
\begin{eqnarray}  \label{eqe1}
e_1=a_{11} \frac{\partial}{\partial s}, \\
\label{eqe2} e_2=a_{22} \frac{\partial}{\partial t}+a_{23} \frac{\partial}{\partial u} \\
\label{eqe3} e_3=a_{32} \frac{\partial}{\partial t}+a_{33} \frac{\partial}{\partial u}.
\end{eqnarray}
By considering \eqref{eqALL} in $\lbrack e_1, e_2] =0$ and $\lbrack e_1, e_3] = -\omega_{13}(e_3)e_3$, we obtain  
\end{subequations}
\begin{equation}  \label{condition1}
a_{11}\Big((a_{22})_s \frac{\partial}{\partial t}+(a_{23})_s \frac{\partial}{%
\partial u}\Big)-\Big(a_{22} \frac{\partial a_{11}}{\partial t}+a_{23} \frac{%
\partial a_{11}}{\partial_u}\Big)=0
\end{equation}
and
\begin{equation}  \label{condition2}
a_{11}\Big((a_{32})_s \frac{\partial}{\partial t}+(a_{33})_s \frac{\partial}{%
\partial u}\Big)-\Big(a_{32} \frac{\partial a_{11}}{\partial t}+a_{33} \frac{%
\partial a_{11}}{\partial u}\Big)= -\omega_{13}(e_3)e_3,
\end{equation}
 respectively. From there, we have 
\begin{subequations} \label{condition4ALL}
\begin{eqnarray}  \label{condition4a}
(a_{22})_s=(a_{23})_s=0, \\ \label{condition4b}
\left( 
\begin{array}{cc}
a_{22} & a_{23} \\ 
a_{32} & a_{33}%
\end{array}%
\right)&\left( 
\begin{array}{c}
(a_{11})_t \\ 
(a_{11})_u%
\end{array}%
\right)=0.
\end{eqnarray}
\end{subequations}
 Since the vectors $e_2$ and $e_3$ are linear independent, the following is satisfied 
\begin{equation*}
\det \left( 
\begin{array}{cc}
a_{22} & a_{23} \\ 
a_{32} & a_{33}%
\end{array}%
\right) \neq 0.
\end{equation*}
So we conclude $(a_{11})_t=0$ and $(a_{11})_u=0$ from \eqref{condition4b}. Therefore, we have 
\begin{subequations}
\begin{equation}  
\label{e1a11} e_1=a_{11}(s) \frac{\partial}{\partial s}.
\end{equation}
By taking consider \eqref{condition4a} in \eqref{eqe2} yields 
\begin{equation}
\label{e2a22} e_2=a_{22}(t,u) \frac{\partial}{\partial t}+a_{23}(t,u) \frac{\partial}{%
\partial u}.
\end{equation}
 Now, we will consider the coordinate change such that $S=\Phi(s), T=\Psi_1(t,u), U=\Psi_2(t,u)$,
 i.e., 
\end{subequations}
\begin{subequations}
\begin{eqnarray}  \label{changes}
\frac{\partial}{\partial s}= \Phi^{\prime }\frac{\partial}{\partial S}, \\
\label{changet} \frac{\partial}{\partial t}= (\Psi_1)_t \frac{\partial}{\partial T}%
+(\Psi_2)_t \frac{\partial}{\partial U}, \\
\label{changeu} \frac{\partial}{\partial u}= (\Psi_1)_u \frac{\partial}{\partial T}%
+(\Psi_2)_u \frac{\partial}{\partial U}.
\end{eqnarray}
 From there, if the transformation \eqref{changes} substitutes in \eqref{e1a11}, then
 we obtain $\displaystyle e_1=a_{11}(s) \Phi^{\prime }(s) \frac{\partial}{%
\partial S}$. On the other hand if the transformations \eqref{changet} and \eqref{changeu} 
consider in \eqref{e2a22}, then we get
\end{subequations}
\begin{equation*}
e_2 = \Big(a_{22} (\Psi_1)_t+ a_{23} (\Psi_1)_u\Big)\frac{\partial}{\partial
T}+ \Big(a_{22} (\Psi_2)_t+ a_{23} (\Psi_2)_u\Big)\frac{\partial}{\partial T}.
\end{equation*}
 Here, if we choose the function $\Phi$ and the functions $\Psi_1,\Psi_2$ as $a_{11}(s) \Phi^{\prime }(s)=1$
and  
\begin{eqnarray*}
a_{22} (\Psi_1)_t+ a_{23} (\Psi_1)_u=1, \\
a_{22} (\Psi_2)_t+ a_{23} (\Psi_2)_u=0,
\end{eqnarray*}
respectively, then the system \eqref{eqALL} is written as
\begin{eqnarray*}
e_1 = \frac{\partial}{\partial S}, \\
e_2 = \frac{\partial}{\partial T} \\
e_3 = {\tilde{a}}_{32} \frac{\partial}{\partial T}+{\tilde{a_{33}}} \frac{%
\partial}{\partial U}.
\end{eqnarray*}
 By abusing the notation in the rest of the proof of the Case 2, we will take as 
$S=s$, $T=t$, $U=u$, ${\tilde{a}}_{32}=a_{32}$ and ${\tilde{a}}_{33}=a_{33}$.

Note that, by considering the above last obtained system and the expressions 
$\widetilde{\nabla}_{e_1}{N}=0$, $\widetilde{\nabla}_{e_1}{e_1}=0$,  $%
\widetilde{\nabla}_{e_2}{e_1}=0$ and $\widetilde{\nabla}_{e_3}{e_1} \neq 0$,
 one conclude $e_1=e_1(u)$ and $N=N(t,u)$. Thus, we can take as $f(s)=s$ 
from the equalities \eqref{GCRE14f1} with an appropriate selection of parameters. 
If these obtained results are substituted in the decomposition \eqref{GCRE14genelnewx}, 
then the position vector $x$ of $M$ is written by 
\begin{equation}  \label{GCRE14genelnewx1}
x(s,t,u) = s e_1(u) - c N(t,u).
\end{equation}
Morever, since the principal curvature corresponding to the tangent vector $e_2$ of 
this hypersurface is satisfied $k_2= \frac{1}{c}$, we get $\displaystyle \widetilde{\nabla}_{e_2}{e_2}=
x_{tt}= -\frac{N}{c}$ with direct calculation. When $x_{tt}= -cN_{tt}$ is considered to in the last expression,
 we get the following partial differential equation
\begin{equation*}
c^2 N_{tt}-N=0
\end{equation*}
 whose solution is given by  
\begin{equation}  \label{case2N}
N(t,u)=F_1(u) \cosh (\frac{t}{c})+ F_2(u) \sinh (\frac{t}{c}).
\end{equation}
 Here since the vector $N$ is the unit normal of hypersurface, the vector valued functions $F_1(u)$ and $%
F_2(u)$ must satisfy 
\begin{equation}  \label{f1f2}
\left\langle F_1,F_1\right\rangle =-1, \quad \left\langle
F_2,F_2\right\rangle =1, \quad \left\langle F_1,F_2\right\rangle =0
\end{equation}
and also $\left\langle N,e_1\right\rangle=0$ is satisfied. On the other hand 
from the expression \eqref{GCRE14genelnewx1}, we get $x_su={e_1}^{\prime}$. Since $x_s$ is 
a principal direction, so one conclude $\left\langle N,x_su\right\rangle =0$. The last obtained these equalities 
considering with \eqref{case2N}, we get 
\begin{equation}  \label{f1f2alpha}
\left\langle F_1,e_1\right\rangle =\left\langle F_2,e_1\right\rangle =0,
\quad \left\langle {e_1}^{\prime},F_1\right\rangle = \left\langle {e_1}%
^{\prime},F_2\right\rangle =0
\end{equation}
where ${\prime}$ shows the ordinary derivative. Thus, the system $\left\{e_1,{e_1}%
^{\prime};F_1,F_2\right\}$ satisfying \eqref{f1f2} and \eqref{f1f2alpha} 
defines an orthonormal field in Minkowski space $\mathbb{E}^4_1$. Note that, since 
$\left\langle e_1,e_1\right\rangle=1$ and $\left\langle e_1,{e_1}^{\prime}\right\rangle=0$ are satisfied, 
so $e_1=e_1(u)$ is really corresponded to the curve $\alpha$ given in Example \ref{GCRE14orn2}. 
Moreover, the system $\left\{F_1,F_2\right\}$ is the normal base of the curve $\alpha$ 
from \eqref{f1f2} and \eqref{f1f2alpha}. Consequently, we get the hypersurface given in the case (ii) of the Theorem.

\textbf{CASE 3.} Let $\nabla_{e_2}{e_1} \neq 0$ and $\nabla_{e_3}{e_1}\neq 0$. 
In this case, we see the vectors $\nabla_{e_2}{e_1}$ and $\nabla_{e_3}{e_1}$ are linear independent
from the equation \eqref{GCRE14DexompofxCase2Eq2c}. Therefore, the description $\sigma$ given by
\begin{eqnarray}
\sigma : \mathbf{D}^{\bot} \longrightarrow \mathbf{D}^{\bot},  \notag \\
\sigma(X) &=& \nabla_{X}{e_1}  \notag
\end{eqnarray}
is the one-to-one. From this statement, the vectors $\displaystyle 
\frac{\partial (e_1)}{\partial t}$ and $\displaystyle \frac{\partial (e_1)}{%
\partial u}$ are linear independent. Morever since $\left\langle
e_1,e_1\right\rangle=1$, 
\begin{eqnarray}
y : \Lambda \longrightarrow \mathbb{S}^3_1(1) \subset \mathbb{E}^4_1,  \notag
\\
y(t,u) &=& e_1(t,u)  \notag
\end{eqnarray}
defines a regular surface. On the other hand by considering $%
\left\langle x_s,N\right\rangle=\left\langle x_t,N\right\rangle=\left\langle
x_u,N\right\rangle=0$, we obtain $\left\langle
y,N\right\rangle=\left\langle y_t,N\right\rangle=\left\langle
y_u,N\right\rangle=0$. So, the vector field $N=N(t,u)$ is the normal of the surface $%
y=y(t,u)$ in the de Sitter space $\mathbb{S}^3_1(1)$. These obtained results substitute 
in the decomposition \eqref{GCRE14genelnewx}, we find the hypersurface given in the case (iii) of the Theorem.


\subsection{Hypersurfaces Lying in the Time-Like Cone}

In this subsection we treat the remaining case: the immersion $x$ of the hypersurface $M$ lies 
in the time-like cone, i.e., $\left\langle x,x\right\rangle =-\mu ^{2}$. In this case, the position 
function $x$ can be decomposed as in \eqref{GCRE14DexompofxCaseII} for smooth functions $\theta,\mu$ 
as before defined. Thus we will give directly the following Lemma for $n=3$, by considering 
\eqref{GCRE14DexompofxCaseIInueqsALL}-\eqref{GCRE14DexompofxCase2Eq2ALL}, \eqref{jij} and \eqref{ijk}:

\begin{lemma} \label{Case2ClassThmDiagonSpacelikekClm11}
The Levi-Civita connection, $\nabla$ of the GCR hypersurface $M$ satisfies  
\begin{subequations} \label{CASEIISpacelikeLeviCivitaEq11ALL}
\begin{eqnarray}
\label{CASEIISpacelikeLeviCivitaEq11a}\nabla _{e_{1}}e_{1}=0, \quad \nabla _{e_{1}}e_{2}=\omega_{23}(e_1)e_3, \quad \nabla_{e_1}e_3=-\omega_{23}(e_1)e_2,
\\\label{CASEIILeviCivitaEq11b}
\nabla _{e_{2}}e_{1}= \omega_{12}(e_2)e_2, \quad \nabla _{e_{2}}e_{2}=-\omega_{12}(e_2)e_1+\omega_{23}(e_2)e_3, \quad \nabla_{e_2}{e_3}=-\omega_{23}(e_2)e_2,
\\\label{CASEIILeviCivitaEq11c}
\nabla _{e_{3}}e_{1}= \omega_{13}(e_3)e_3, \quad \nabla _{e_{3}}e_{2}=-\omega_{23}(e_3)e_3, \quad \nabla_{e_3}{e_3}=-\omega_{13}(e_3)e_1-\omega_{23}(e_3)e_2.
\end{eqnarray}
\end{subequations}
Here, $\displaystyle \omega_{1l}(e_l)=\frac{1-\mu \sinh \theta k_l}{\mu \cosh \theta}$
such that $l=2,3$. Also, the principal curvatures $k_1,k_2,k_3$ satisfy 
\begin{subequations} \label{ClassThmDiagonSpacelikekCod1Case12ALL}
\begin{eqnarray} 
\label{ClassThmDiagonSpacelikekCod1Case121a}
\omega_{23} (e_1) (k_2-k_3) =0, &&
\\\label{ClassThmDiagonSpacelikekCod1Case121b}
e_2(k_1) = e_3(k_1) = 0, &&
\\\label{ClassThmDiagonSpacelikekCod1Case121c}
e_1(k_2) = \omega_{12}(e_2)(k_1-k_2), &\quad e_3(k_{2}) = \omega_{23}(e_2)(k_2-k_3),
\\\label{ClassThmDiagonSpacelikekCod1Case121d}
e_1(k_{3}) = \omega_{13}(e_3)(k_1-k_3), &\quad e_2(k_3) = \omega_{23}(e_3)(k_2-k_3), 
\end{eqnarray}
\end{subequations}
for $i,j,k=1,2,3$.
\end{lemma}

Now, we would like to give the following examples of GCR hypersurfaces with vanishing 
Gauss-Kronecker curvature lying in the time-like cone:

\begin{example} \label{GCRE14timelikeorn1}
Let $M = \mathbb S^2(\frac{1}{c^2}) \times \mathbb E^1$ in Minkowski spaces $\mathbb E^4_1$ be a hypercylinder parametrized as 
$$x(s,t,u)=\Big(x_1(s,t),x_2(s,t),x_3(s,t),u\Big).$$
Further, its normal is $N(s,t)=\Big(x_1(s,t),x_2(s,t),$ $x_3(s,t),0\Big)$. Therefore, \linebreak
$\displaystyle x = u \frac{\partial}{\partial u}+ N$. Note that, the tangent vector $\displaystyle%
\frac{\partial}{\partial u}$ is the principal direction corresponding with the principal curvature $k_1=0$.
Consequently, the hypercylinder $M$ with vanishing Gauss-Kronecker curvature is a GCR hypersurface.
\end{example}

\begin{example} \label{GCRE14timelikeorn2}
Let $\beta$ be a unit speed curve lying on $\mathbb H^3_1(-1) \subset \mathbb E^4_1$ 
and $V_1(u),V_2(u)$ two orthonormal vector fields spanning the normal space of $\beta$, i.e., 
\begin{eqnarray}
\left\langle V_1,V_2\right\rangle=\left\langle V_1,\beta \right\rangle=\left\langle V_1,\beta^{\prime}\right\rangle=\left\langle V_2,\beta \right\rangle=\left\langle V_2, \beta^{\prime}\right\rangle=0, \\
\left\langle V_1,V_1\right\rangle=\left\langle V_2,V_2\right\rangle=1.
\end{eqnarray}
Consider the hypersurface in $\mathbb E^4_1$ given by 
\begin{equation} \label{GCRE14timelikex1}
x(s,t,u)=s \beta(u)-c \Big(V_1(u) \cosh \left(\frac{t}{c}\right)+V_2(u) \sinh \left(\frac{t}{c}\right)\Big)
\end{equation}
for a non-constant $c$. Then, by a direct computation, one can check that the unit normal vector
field of $M$ is $N= V_1(u) \cosh \left(\frac{t}{c}\right)+V_2(u) \sinh \left(\frac{t}{c}\right)$
 and principal directions of $M$ are $e_1= \partial_s =\beta(u), \quad e_2= \partial_t,%
 \quad e_3= \frac{1}{\left\|x_u\right\|} \partial_u $ corresponding to principal curvatures $0,\frac{1}{c}, k_3$, respectively.
 Thus, the hypersurface $M$ is a GCR with vanishing Gauss-Kronecker curvature because of the first principal curvature $k_1=0$.
\end{example}

\begin{example} \label{GCRE14timelikeorn3}
Let $r : \Omega \longrightarrow \mathbb H^3_1(-1) \subset \mathbb E^4_1$
 be an oriented regular surface with the spherical normal $N$, where 
$\Omega$ is an open subset in $\mathbb R^2$. Consider the hypersurface $M$ in $\mathbb E^4_1$ given by
\begin{eqnarray} 
\nonumber x: I \times \Omega \longrightarrow \mathbb E^4_1, \\
\label{GCRE14timelikex2} x(s,t,u)=s r(t,u)-c N(t,u)
\end{eqnarray}
where $c$ is a constant. Here, since the vector field $N$ is the spherical normal of the surface $r$, 
$$\left\langle r_t,N\right\rangle = \left\langle r_s,N\right\rangle = \left\langle r,N\right\rangle = 0$$
are satisfied. Thus, by considering the above equalities in \eqref{GCRE14timelikex2}, we get directly 
$$\left\langle x_t,N\right\rangle = \left\langle x_s,N\right\rangle = \left\langle x_u,N\right\rangle = 0.$$
 So, one can conclude that the vector field $N$ is also the unit normal of the hypersurface of $M$.
 Furthermore, one get $x_{ss}=0$ from \eqref{GCRE14timelikex2} and also
$\left\langle x_{st},N\right\rangle =0$ and $\left\langle x_{tt},N\right\rangle =0$. 
 So, $h(\partial_s,X)=0$ satisfied for all tangent vector $X$ on $M$ which say that $S(\partial_s)=0$.
 Therefore, the Gauss-Kronecker curvature of the hypersurfece $M$ vanishes and 
$\partial_s=r(t,u)$ is a principal direction of the hypersurface. Consequently, $M$ is a GCR hypersurface.
\end{example}

Finally, we want to give our second main theorem that his proof is completely similar to 
the proof of the first main theorem as Theorem \ref{GCRE14Spacelikekoni} given in the previos subsection:

\begin{theorem}
Let $M$ be a space-like hypersurface with vanishing Gauss-Kronecker curvature
 in the Minkowski space $\mathbb E^4_1$. Then, $M$ is a GCR hypersurface lying in the time-like cone
 if and only if it is congruent to one of the following 3 types of hypersurfaces.
\begin{enumerate}
\item [(i)] A part of the hypercylinder given in Example \ref{GCRE14timelikeorn1},

\item [(ii)] A hypersurface parametrized with \eqref{GCRE14timelikex1} in Example \ref{GCRE14timelikeorn2},

\item [(iii)] A hypersurface parametrized with \eqref{GCRE14timelikex2} in Example \ref{GCRE14timelikeorn3}.
\end{enumerate}
\end{theorem}


\section*{Acknowledgements}

This paper is a part of PhD thesis of the first named author who was supported by The Scientific and
Technological Research Council of Turkey (TUBITAK) as a PhD scholar.


\end{document}